\documentclass[12pt,a4paper]{article}
\usepackage[utf8]{inputenc}
\usepackage{tikz-network}
\usepackage{pgfplots}
\usepackage{amsmath}
\usepackage{amssymb}
\usepackage{mathtools}
\usepackage{amsthm}
\usepackage{authblk}
\usepackage{graphics}
\usepackage{graphicx}
\usepackage{tikz}
\usetikzlibrary{shapes.geometric, arrows}
\usepackage{enumitem}
\usepackage{pifont}
\usepackage[export]{adjustbox}
\usepackage{geometry}
\usepackage{hyperref}
\newcounter{thm}

\newtheorem{therm}[thm]{Theorem}

\newtheorem{lem}[thm]{Lemma}
\newtheorem{cor}[thm]{Corollary}

\usetikzlibrary{shapes,arrows,positioning}
\title{ On g-Extra Connectivity of Corona-Type Graph Products.
}
\author[a]{Arati Sharma\thanks{Email: sharmaarati28@gmail.com}}
\author[a]{Satyam Guragain}
\author[a]{Ravi Srivastava\thanks{Corresponding author, Email: ravi@nitsikkim.ac.in}}
\affil[a]{Department of Mathematics, National Institute of Technology Sikkim, Sikkim 737139, India}
\date{}

\pgfplotsset{compat=1.18}
\begin{document}
\maketitle

\begin{abstract}
\noindent Connectivity is one of the central ideas in graph theory, especially when it comes to building fault-tolerant networks. A cutset $S$ of $G$ is defined to be the set of vertices in $G$ whose removal disconnects the graph. An $R_g$ cutset of $G$ is a cutset whose removal disconnects the graph in such a way that each connected component has at least $g+1$ vertices. If $G$ has at least one $R_g$ cutset then the $g-$extra vertex connectivity (or the $g-$extra edge connectivity), denoted as $\kappa_g(G)$ ($\lambda_g(G)$), is defined as the minimum cardinality of $R_g$ cutset. In this paper, we obtain the $g-$ extra connectivity of various corona type graph products edge corona, neighbourhood corona, subdivision vertex neighbourhood corona,subdivision edge neighbourhood corona and generalised corona product.
\end{abstract}
\section{Introduction}
All graphs considered in this paper are connected, simple, and finite.
Let $G=(V,E)$ be a graph with vertex set $V=\{v_1,v_2,...,v_n \}$ and edge set $E(G)=\{e_1,e_2,...,e_n\}$.The degree of a vertex $v$ in a graph $G$, denoted as $deg_G(v)$, is the number of edges incident to $v$. The neighbourhood of a vertex $v$ in a graph $ G$, denoted by $N_G (u)$, is the set of all vertices adjacent to $u$. The minimum degree of graph $G$, denoted as $\delta(G)$, is the smallest degree among all $v$. A separating set or vertex cut of a graph $G$ is a set $S\subseteq V(G)$ such that $G-S$ has more than one component. The connectivity of $G$, written as $\kappa(G)$, is the minimum size of a vertex set such that $G-S$ is disconnected. An edge cut is a set $F\subseteq E(G)$ such that $G-F$ has more than one component. The edge-connectivity of $G$, written as $\lambda(G)$, is the minimum size of an edge set such that $G-F$ is disconnected. The $g-$extra vertex connectivity (resp.  $g-$extra edge connectivity) of a graph is the set $S\subseteq V(G)$(resp. $S \subseteq E(G)$ whose removal disconnects the graph and each component has at least $g+1$ vertices, denoted as $\kappa_g(G)$ (resp. $\lambda_g(G)$). If no such set exists, we define $\kappa_g(G)= \infty $ ( resp. $\lambda(G)=\infty$). 
Define $H^i$ as the copy of $H$ corresponding to the $i^{th}$ vertex of $G$. Define $H_{e_{ij}}$ as the copy of $H$ corresponding to edge $e_{ij}$. The subdivision graph of a graph $G$, denoted by $S(G)$, is the graph obtained by subdividing every edge of $G$ exactly once, that is, for each edge $e = uv \in E(G) $, a new vertex is inserted into $e$, and the edge $e$ is replaced by two edges $ (u, w)$and $(w, v)$, where $ w $ is the newly inserted vertex. Clearly, $\kappa_0(G)= \kappa(G) $ for any connected non-complete graph $G$.
In \cite{wang2019g}, Wang et al.in their study investigated the $g$-extra connectivity of various graph types. The graphs for which $\kappa_g(G) = 1$, $2$, or $3$. For trees with $n$ vertices, they demonstrated that the $g$-extra connectivity is given by $\kappa_g(T_n) = n - 2g - 2$. In \cite{wang2024g}, Wang et al. derived the precise values of $ g-$extra connectivity of the lexicographic product, and Cartesian product, along with some applications. Zhu and Tian \cite{zhu2023extra} studied $g(\geq 3)$extra connectivity for a strong product of $G_i$ $i=1,2$, where $G_i$ is a maximally connected regular graph. Gu and Hao \cite{gu20143} deduced the $3-$extra connectivity of $3-$ary-$n-$ cubes as a continuation of the work of Hsieh and Chang\cite{hsieh2012extraconnectivity} who obtained $2-$ extra connectivity for k-ary-n cubes $Q_n^k$ for $k\geq 4$ and $n\geq5$. Zhu et al. \cite{zhu2011extra}obtained $1-$extra connectivity and $2-$extra connectivity for$ 3-$ary-$n-$cube. Chang et al. \cite{chang20133} obtained $3-$extra connectivity (resp. $3-$extra edge connectivity) of an n-dimensional folded hypercube and provided an upper bound for $g-$extra connectivity on folded hypercubes for $g\geq 6$. Further, Zhang and Zhou \cite{zhang2015g} extended $g-$extra connectivity of $n-$dimensional folded hypercubes for $0 \leq g \leq n + 1$ and $n\geq 7$. Whang et al. \cite{wang2018g} explored the $g-$extra connectivity and $g-$extra diagonalisability of the crossed cube. Li et al. \cite{li2021extra} obtained the $g-$extra edge connectivity for some special graphs and determined the upper and lower bounds of $g-$extra edge connectivity. In this paper, we explore the $g-$extra vertex connectivity and $g-$ extra edge connectivity of several graph products. The rest of the paper is organized as follows. In section \ref{1} we derive the $g-$ extra vertex connectivity and $g-$ extra edge connectivity of corona-type products.
\section{Application and Background}
A multiprocessor system (MPS) comprises of an autonomous processor with local memory interconnected via communication links. When designing a multiprocessor system it is presumed that any group of components, such as processors or links, could fail at once. One of the primary considerations while designing MPS is its system level fault tolerance. Fault refers to processor or link failures ans a system is said to be fault-tolerant if it can continue to function in spite of some failure. The architecture of a MPS is typically represented by a graph $G =(V,E)$ where $V$ denotes the set of processors and $E$ denotes the set of communication links connecting the processors. Two functionality criteria have been widely studied by \cite{esfahanian2002generalized}. One of these criteria for the system to be considered functional is whether the network logically retains the topological structure. This issue arises when one design is embedded within another, the presence of specific topology plays a crucial role in providing the required performance. The second criteria claims the system is functional if there is a fault-free communication path between each pair of non-faulty nodes. Hence, connectivity plays a major role when it comes to building fault-tolerant networks.
\section{The $g-$ extra vertex connectivity and $g-$ extra edge connectivity of corona-type products} \label{1}
We consider the base graph $G$ as a connected,non-complete graph in all the products discussed below.
\subsection{Edge corona}\label{sec1}
Definition: Let $G$ and $H$ be two graphs such that $G$ has $n_1$ vertices and $m_1$ edges, while $H$ has $n_2$ vertices and $m_2$ edges. The edge corona of $G$ and $H$, denoted by $G \diamond H$, is the graph obtained by taking one copy of $G$ along with $m_1$ separate copies of $H$ such that each edge $e_i = u_iv_i \in E(G)$, is connected to every vertex in the $i$-th copy of $H$.
\begin{lem}
    \cite{wang2019g} Let $g$ be a non-negative integer, and let $G $ be a connected graph. Then,\begin{equation*}
  \kappa_g(G) \leq \kappa_{g+1}(G).
    \end{equation*}
        \end{lem} 
\begin{lem}
    \cite{wang2019g} Let $g$ be a non-negative integer and let $G$ be a connected graph of order $n$ such that $0 \leq g \leq \left\lfloor \frac{n-\kappa(G)-2}{2} \right\rfloor
 $.Then,
    \begin{equation*}
        \kappa(G)\leq \kappa_g(G) \leq n-2g-2.
    \end{equation*}
\end{lem}
\begin{lem} \cite{li2021extra} Let $g$ be a non-negative integer and let $G$ be a connected graph. Then \begin{equation*}
 \lambda_g(G)\leq \lambda_{g+1}(G).   
\end{equation*}
\end{lem}
\begin{lem}\label{lem1}
Let $G$ be a connected graph of order at least $3$ and let $H$ be a connected graph of order $m$. 
Let $k$ be a maximum integer such that there exists a minimum vertex cut $A \subseteq V(G)$ with the property that every component of $G-A$ contains at least $k+1$ vertices. 
Then, for $0 \leq g+1 \leq (k+1) + km + \delta(G) m$, we have
\begin{equation*}
 \kappa_{g}(G \diamond H) = 1 \quad \text{if and only if} \quad \kappa(G) = 1.   
\end{equation*}
\end{lem}
\begin{proof}

Let $\kappa_{g}(G\diamond H)=1$, that is, there exists a vertex $u  \in G\diamond H$ whose removal disconnects the graph and each connected component has at least $g+1$ vertices. We prove $u\in G$. Assume to the contrary that $u \not\in G $ which implies $u \in H$, but removing a vertex from $H$ will never disconnect $G \diamond H$ as the remaining $m-1$ vertices of $H$ will be connected to vertices of $G$. Therefore, $u \in G$ . Hence $G-u$ is disconnected, and since $G$ is connected, it implies $\kappa(G)=1.$ 
\\Conversely, let $\kappa(G)=1$ then there exists a vertex $v$ that disconnects the graph $G$. By definition of $G \diamond H$ , $G \diamond H- v$  is also disconnected with at least $(k+1)+km+\delta(G) m$ vertices, so at least $g+1$ vertices.  $\kappa_{g}(G \diamond H)=1$.

\end{proof}
\begin{therm}
  Let $G$ and  $H$ be two graphs of orders $n$ and  $m$ respectively such that $n\geq3$
  
      (i) If $0\leq g\leq m-1$, then  $\kappa_g(G \diamond H)= 1$ or $2$.
      
(ii) If $k (km+1) < g+1 \leq (k+1)+km+\delta(G) m$, then \begin{equation*}
   \kappa_g (G \diamond H)=
   \begin{cases}
     |A|  &\text{ when } A \text{ does not contain any adjacent vertices. } \\
    |A|+|A^{'}|m & \text{ when } A\text{ contains $A^{'}$ adjacent pairs}.
\end{cases} 
\end{equation*} 
where $A$ is defined to be the minimum cutset and  $G-A$   has at least $k+1 $ vertices in each component.
  \begin{proof}

 Let $V(G) =\{u_1,u_2,...,u_n\}$. It is clear $G\diamond H - (u_i,u_{i+1})$ is disconnected where $u_i \sim u_{i+1}$ and every connected component has at least $m$ vertices i.e., $(g+1)$ vertices hence, $\kappa_g(G \diamond H) \leq 2$. Since, $G \diamond H$ is a connected graph at least one vertex has to be removed to disconnect the graph, so $1\leq \kappa(G \diamond H )\leq \kappa_g(G \diamond H)\leq 2$. By lemma \ref{lem1},$\kappa_{g}(G \diamond H)=1$ if and only if $\kappa(G)=1$. If $\kappa(G)\geq2$, then $1 < \kappa_g(G \diamond H)\leq 2 $. Therefore we conclude $\kappa_g(G \diamond H)=2$. This proves $(i)$. To establish part (ii), we now consider two distinct cases.\\
     Case I: $A$ does not contain any adjacent vertices. By the construction of $G \diamond H $ and  definition of $A$ ,$G \diamond H -A$ is disconnected and each connected component has  $(k+1)+km+\delta(G)m $ vertices i.e., at least $(g+1)$ vertices and hence, $\kappa_g(G \diamond H) \leq |A|$.
      It suffices to prove $\kappa_g(G \diamond H) \geq |A|$.By definition of $\kappa_g(G \diamond H)$ there exists a cutset say  $T$ such that $\kappa_g(G \diamond H)=|T|$ where$(G \diamond H)-T$ is not connected and each connected component of $(G \diamond H)-T$ has at least $g+1$ vertices.\\
 Claim: $|T|\geq |A|$. \\
    Suppose the claim does not hold $|T|< |A|$. $G-T$ is not connected otherwise, we will get a component with fewer than $g+1 $ vertices, a contradiction. This further violates the assumption that $A$ is a minimum cutset of G. Therefore,$|T|\geq |A|$. Hence,$\kappa_g(G \diamond H)=|A|$.\\
        Case 2: When $A$ has adjacent vertices. Based on the definition of $A$, $G-A$ is disconnected and each connected component of $G-A$ has at least $(k+1)$ vertices. Let $|A|=x$ and $|A^{'}|=x^{'}$define $S= A \cup V(H_{e_{ij}})$where $e_{ij} \in E(A)$ , $|S|= x+x^{'}m$. $(G\diamond H)-S$ is disconnected and each connected component has atleast $(k+1)+km+\delta(G) m$ vertices i.e., at least $(g+1)$ vertices which implies $\kappa_g(G \diamond H) \leq |S| $. Let $\kappa_g(G \diamond H)=|T|$ such that $(G \diamond H)-T$ is not connected and each connected component of $(G \diamond H)-T$ has at least $g+1$ vertices.\\ Claim:$|T \cap V(G) | \geq |A|.$ Let if possible $|T \cap V(G) | < |A|.$ $G-T$ is disconnected so there exists a component $C_1$ with atmost $k$ vertices with at least $(k-1) $edges and atmost $\frac{k(k-1)}{2}$ edges. Hence, $C_1 \cup V(H_{e_{ij}})$ where $e_{ij} \in C_1$ is a connected component in $(G \diamond H)-T$  with fewer than $g+1$ vertices, a contradiction which implies $|T\cap V(G) | \geq |A|$. Since $T \cap V(G)  \subseteq T$ also the $H$ copy corresponding to edges in $T\cap V(G)$ will be contained in $T$. $|T|\geq |T\cap V(G)|(m+1) \geq |A| (m+1)\geq |A|+|A^{'}|m=|S|$. 
  \end{proof}
  \end{therm}
Note: The upper bound is not strict, there might exists cases when the number of vertices in each component exceeds $(k+1)+km+\delta(G) m$. This arises from the possibility that some vertices in each component may have degree more than
$\delta(G)$. An exact count can only be determined for $r- $regular graphs.
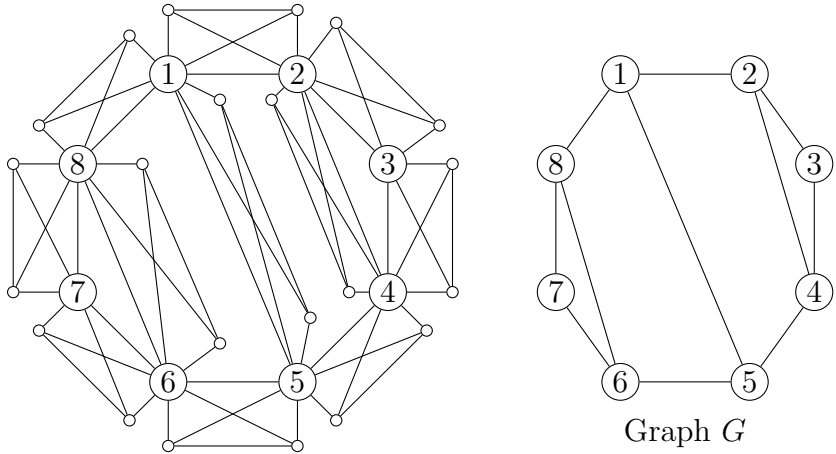
\begin{figure}[ht]
         \centering
         \begin{tikzpicture}[scale=0.17]

        \node (1) at (-5,12) [circle, draw, inner sep=1.5pt] {1};
        \node (2) at (5, 12) [circle, draw, inner sep=1.5pt] {2};
        \node (3) at (12,5) [circle, draw, inner sep=1.5pt] {3};
        \node (4) at (12, -5) [circle, draw, inner sep=1.5pt] {4};
        \node (5) at (5,- 12) [circle, draw, inner sep=1.5pt] {5};
        \node (6) at (-5, -12) [circle, draw, inner sep=1.5pt] {6};
        \node (7) at (-12, -5) [circle, draw, inner sep=1.5pt] {7};
        \node (8) at (-12, 5) [circle, draw, inner sep=1.5pt] {8};

        \draw (1) -- (2);
        \draw (2) -- (3);
        \draw (3) -- (4);
        \draw (4) -- (5);
        \draw (5) -- (6);
        \draw (6) -- (7);
        \draw (7) -- (8);
        \draw (1) -- (8);
        \draw (6) -- (8);
        \draw (1) -- (5);
        \draw (2) -- (4);
     \node (a) at (30,12) [circle, draw, inner sep=1.5pt] {1};
   \node (b) at (40, 12) [circle, draw, inner sep=1.5pt] {2};
       \node (c) at (45,5) [circle, draw, inner sep=1.5pt] {3}; 
    \node (d) at (45, -5) [circle, draw, inner sep=1.5pt] {4};   
     \node (e) at (40,- 12) [circle, draw, inner sep=1.5pt] {5};
    \node (f) at (30, -12) [circle, draw, inner sep=1.5pt] {6};
    \node (g) at (25, -5) [circle, draw, inner sep=1.5pt] {7};
    \node (h) at (25, 5) [circle, draw, inner sep=1.5pt] {8};
         \draw (a) -- (b);
        \draw (b) -- (c);
        \draw (c) -- (d);
        \draw (d) -- (e);
        \draw (e) -- (f);
        \draw (f) -- (g);
        \draw (g) -- (h);
       \draw (h) -- (a);
        \draw (a) -- (e);
        \draw (b) -- (d);
         \draw (h) -- (f);

        \node at (35, -16) {Graph $G$};

\node (11) at (-5,17) [circle, draw, inner sep=1.5pt] {};
\node (12) at (5,17) [circle, draw, inner sep=1.5pt] {};
\node (21) at (8,16) [circle, draw, inner sep=1.5pt] {};
\node (22) at (16,8) [circle, draw, inner sep=1.5pt] {};

      \node (31) at (17,5) [circle, draw, inner sep=1.5pt] {};
        \node (32) at (17, -5) [circle, draw, inner sep=1.5pt] {};
         
\node (41) at (15, -8) [circle, draw, inner sep=1.5pt] {};
        \node (42) at (8,- 15) [circle, draw, inner sep=1.5pt] {};

    \node (51) at (-5,-17) [circle, draw, inner sep=1.5pt] {};
        \node (52) at (5, -17) [circle, draw, inner sep=1.5pt] {};

\node (81) at (-15, 8) [circle, draw, inner sep=1.5pt] {};
        \node (82) at (-8,15) [circle, draw, inner sep=1.5pt] {};

 \node (71) at (-17,5) [circle, draw, inner sep=1.5pt] {};
        \node (72) at (-17, -5) [circle, draw, inner sep=1.5pt] {};
\node (61) at (-15, -8) [circle, draw, inner sep=1.5pt] {};
        \node (62) at (-8,- 15) [circle, draw, inner sep=1.5pt] {};

\node (91) at (-1,10) [circle, draw, inner sep=1.5pt] {};
        \node (92) at (6, -7) [circle, draw, inner sep=1.5pt] {};

\node (101) at (3,10) [circle, draw, inner sep=1.5pt] {};

 \node (102) at (9, -5) [circle, draw, inner sep=1.5pt] {};
  \node (111) at (-7, 5) [circle, draw, inner sep=1.5pt] {};
    \node (112) at (-1, -9) [circle, draw, inner sep=1.5pt] {};

        \draw (11) -- (12);
        \draw (21) -- (22);
        \draw (31) -- (32);
        \draw (41) -- (42);
        \draw (51) -- (52);
        \draw (61) -- (62);
        \draw (71) -- (72);
        \draw (81) -- (82);
         \draw (91) -- (92);
          \draw (101) -- (102);
          \draw (111) -- (112);

        \draw (11) -- (1);
        \draw (11) -- (2);
        \draw (12) -- (1);
        \draw (12) -- (2);
         \draw (21) -- (2);
        \draw (21) -- (3);
        \draw (22) -- (2);
        \draw (22) -- (3);
         \draw (31) -- (4);
        \draw (32) -- (4);
        \draw (31) -- (3);
        \draw (32) -- (3);
         \draw (41) -- (4);
        \draw (42) -- (4);
        \draw (41) -- (5);
        \draw (42) -- (5);
         \draw (51) -- (6);
        \draw (52) -- (6);
        \draw (51) -- (5);
        \draw (52) -- (5);
         \draw (61) -- (7);
        \draw (62) -- (7);
        \draw (61) -- (6);
        \draw (62) -- (6);
         \draw (71) -- (8);
        \draw (72) -- (8);
        \draw (71) -- (7);
        \draw (72) -- (7);
         \draw (81) -- (1);
        \draw (82) -- (1);
        \draw (81) -- (8);
        \draw (82) -- (8);
         \draw (111) -- (8);
        \draw (112) -- (8);
        \draw (111) -- (6);
        \draw (112) -- (6);
         \draw (91) -- (1);
        \draw (92) -- (1);
        \draw (91) -- (5);
        \draw (92) -- (5);
        \draw (101) -- (2);
        \draw (102) -- (2);
        \draw (101) -- (4);
        \draw (102) -- (4);

\end{tikzpicture}
         \caption{Example of $G \diamond K_2.$}
         \label{Edge corona}
     \end{figure}\\
   Example: In the figure \ref{Edge corona}, we observe that the minimum degree $\delta(G) = 2$.Suppose $A=\{ 1,5\}$,then $G-A$ has two components with atleast $3$ vertices and also $G\diamond H-A$ has two component. Based on our general bound, the value of $g + 1$ is upper bounded by $11$ vertices. However, due to the specific structure of the graph, we can extend $g + 1$ up to $13$ vertices. This variation arises because some components of $G - S$ includes vertices of degree greater than $\delta(G)$, which allows for a higher value of $g + 1$.
   \begin{lem}
       Let $G$ and $H$ be two connected graphs, and $F\subseteq E(G)$. Then, $G\diamond H\smallsetminus F$ is connected.
       \begin{proof}
          Each $e=(u,v) \in E(G)$ corresponds to a copy  $H_e $ in $G \diamond H$ which has every vertex adjacent to both  $u,v$. Removing edges from $G$ will not disconnect $G \diamond H$ as there always exists a path $u-x_i -v$, $x_i \in V(H). $ Hence every adjacency of $G$ is preserved in $G\diamond H$ even after deleting edges from $G-$part which implies no vertices are isolated therefore $G \diamond H$ is connected.
       \end{proof}
   \end{lem}
\begin{therm}
    Let $G$and $H$ be two connected graphs with $n,m $ vertices respectively with $n\geq3$, $m\geq 2$. If $0\leq g\leq m-1$ then \begin{equation*}
   \lambda_g(G \diamond H)= \begin{cases}
    m+1  &\text{ when }  \lambda(G)=1 \\
    2m & \text{ when }\lambda(G)\geq1  .
\end{cases} 
\end{equation*}
\begin{proof}
    Let $e=(u_i,u_j) \in E(G)$. Let $u_i $ be a non pendant vertex. Clearly, $u_i$ corresponds to $(d(u_i)+m)$ edges in $G \diamond H$. $G\diamond H- E[u_i,V(H)] $will remove all the $m$ edges attached to $u_i$, additionally removing $e$ will result in disconnected graph with at least $g+1 $ vertices. Hence, $\lambda_g(G\diamond H)\leq m+1.$ Observe removing $m$ edges will not disconnect the graph as every vertex in $G \diamond H $ is connected to atleast $\delta(G)+m$ vertices. In order to isolated a vertex atleast $\delta(G)+m$ edges has to be removed.$\delta(G)+m \leq \lambda_g(G\diamond H)$. Since, $G$ is connected $1+m \leq \lambda_g(G\diamond H)$.\\
    Since, $m $ vertices of $H$ are connected to two vertices each in $G$, which corresponds to $2m $ edges in $G \diamond H$. Removing these $2m$ edges disconnect the graph into components with at least $(l+1)$ vertices $\lambda_g(G\diamond H)\leq 2m.$ Removing less than $2m$ edges in $G \diamond H$ will keep the graph connected as we remove edges from E[V(H), G] so at least $2m$ edges $2m\leq \lambda_g(G\diamond H)$.
\end{proof}
\end{therm}
In the proof above, we treat $u_i$ as a non-pendant vertex because removing all edges from the pendant vertex will isolate it, reducing the case to only $g=0$. 
\subsection{Neighbourhood corona }
  Definition: Let $G$ and $H$ be two graphs, where $G$ has $n$ vertices and $H$ has $m$ vertices. The neighbourhood corona of $G$ and $H$, denoted by $G * H$, is the graph obtained by taking one copy of $G$ and $n$ copies of $H$. For each vertex $u_i \in V(G)$, every vertex in the $i$-th copy of $H$ is connected to all neighbors of $u_i$ in $G$.\\
  \begin{therm}
      Let $G$ be a connected graph $n\geq5$ and $H$ be a connected graph with $m$ vertices.\\
      (i) If $1\leq g \leq m-1$ then, $\kappa_g(G*H)=\delta(G)$\\ 
    (ii) If $k(m+1)< g+1 \leq (k+1)(m+1)$ then, $\kappa_g(G*H)=|A| (m+1)$ where $A$ is the minimal disconnecting set of $G$ such that $G-A$ has atleast $(k+1)$ vertices where   {$k>=1$}.
    \begin{proof}
        Let $V(G)=\{ u_1, u_2,...,u_n\}$, suppose $deg(u_k)=\delta (G)$ is the minimum degree of $G.$ Let neighbourhood of $u_k$ be denoted as $N(u_k)= \{ u_{k_1},u_{k_2},...,u_{k_\delta}\}$.The graph $(G*H)-N(u_k)$ is disconnected and each connected component has at least $m$ vertices which implies $\kappa_g(G*H) \leq |N(u_k)|=\delta(G)$. To disconnect $G \diamond H$  at least all neighbourhood of a vertex has to be removed,  to achieve the minimum we remove $\delta(G) $vertices, if we remove $\delta(G)-1$ vertices then $H_{u_k}$ is still connected to that one remaining vertex keeping $G*H$ connected.\\ Hence, $\delta(G) \leq \kappa(G*H) \leq \kappa_g(G*H)$. This proves part(i).\\
        From the definition of $A$, $G-A $ is not connected, and each connected component has at least $(k+1)$ vertices. Define, $T= A \cup \{ \cup_{i\in A}H^{i}\}$ this implies $|T|=x+xm=|A|(m+1)$.The graph $G*H- T$ is disconnected graph with atleast $(k+1)(m+1) $ vertices,
        $\kappa_g(G*H)\leq |T|.$It is enough to show that $\kappa_g(G*H)\geq |T|.$ By definition of $\kappa_g(G*H)$ there exists $Y \subseteq V(G*H)$ such that $\kappa_g(G*H)=|Y|$such that $(G*H)-Y$ is not connected and each connected component has at least $g+1$ vertices.\\
        Claim:$|Y \cap V(G)| \geq |A|$. Suppose the statement does not hold, then there exists a connected component $C_1$ of $G-Y$ with at most $k $ vertices.Then, $C_1 \cup \{ \cup_{i\in V(C_1)}H^{i}\}$ is connected component of $(G*H)-Y$ with atmost $k(m+1)$ vertices contradiction. Let $|Y \cap V(G) | =\{u_{a_1},...,u_{a_z}\}$ since,$Y \cap V(G)  \subseteq Y$ also the $H$ copy corresponding vertices in $Y \cap V(G)$ will be contained in $Y$. $|Y\cap V(G)|(m+1) \leq |Y|$ hence, $|A|(m+1) \leq |Y \cap V(G)|(m+1) \leq |Y|$.This proves, $\kappa_g(G*H)=|A|(m+1)$.  
    \end{proof}
  
  \end{therm}
        \subsection{Subdivision neighbourhood coronae}
Let $G$ and $H$ be two connected graphs. The \emph{subdivision vertex neighbourhood corona} of $G$ and $H$, denoted by $G \boxdot H$, is defined as the graph obtained by first constructing the subdivision graph $S(G)$ of $G$, then associating a distinct copy $H_i$ of $H$ with each vertex $v_i \in V(G)$, and finally joining every neighbour of $v_i$ in $S(G)$ to every vertex of the corresponding copy $H_i$. The subdivision edge neighbourhood corona of $G$ and $H$, denoted by $G \boxminus H$, is the graph obtained by taking the subdivision graph $S(G)$ of $G$, then associating a distinct copy $H_i$ of $H$ for each edge  $e_i=(u,v) \in E(G)$, and finally joining every neighbour of $v_{e_i}$ in $S(G)$ to every vertex of the corresponding copy $H_i$.\\ 
        \begin{lem}
            Let $G$ and $H$ be two connected non-trivial graphs, then $\kappa(G \boxdot H)= \delta(G).$
            \begin{proof}
              Let $V(G)=\{u_1,u_2,...,u_n\} $. For each vertex $u_i \in V(G)$, a corresponding copy of $H$, denoted by $H_{u_i}$, is inserted. For each edge $\{u_i, u_j\} \in E(G)$, a subdivision vertex is added and connected to the corresponding copies $H_{u_i}$ and $H_{u_j}$.
 .Consider a vertex $u_i$ such that $deg(u_i)=\delta(G).$Removing $N_{S(G)}(u_i)=\{u_{i_1},u_{i_2},...,u_{i_\delta(G)}\}$ will disconnect the graph $G \boxdot H$. Hence,$\kappa(G \boxdot H) \leq \delta(G)$. \\
              Claim: $\kappa(G \boxdot H) \geq \delta(G)$.\\ 
               Removing only vertices in $H$ will not disconnect $G \boxdot H $.In order to disconnect $G \boxdot H $, vertices of $G$ has to be removed. So, any minimum cutset of $G \boxdot H $ does not contain any vertex of $H$ otherwise the cutset would contain $(m-1)+\delta(G)$ vertices. Any $H_{u_i}$ copy is connected to at least $deg(u_i)\geq \delta(G)$ subdivision vertices. Removing fewer than $\delta (G)$ vertices cannot remove all vertices of $H_{u_i}$ attached to subdivision vertex, so we have to remove at least $\delta(G)$ vertices $\kappa(G\boxdot H)\geq \delta(G)$.
              \end{proof}
        \end{lem}
        \begin{cor}
    For any connected graphs $G$ and $H$ $\kappa_g(G \boxdot H)=\delta(G) $ iff $g=0$
\end{cor}
\begin{proof}
    Immediately follows from the above lemma.
\end{proof}
\begin{therm}
    Let $G, H$ be two connected graphs with $n\geq 3 $ and $m$ vertices respectively.\\
    (i) If $1\leq g \leq m-1$,then $\kappa_g(G \boxdot H)=\delta(G)+1. $\\
    (ii) If $k(m+1)\leq g+1 \leq (k+1)(m+1)+\delta(G)$ then, \begin{equation*}
   \kappa_g (G \diamond H)= \begin{cases}
     |A|(m+1)  &\text{ when } A \text{ does not contain any adjacent vertices. } \\
    |A|(m+1)+|A^{'}| & \text{ when } A\text{ contain adjacent vertices and has }  A^{'} \text {adjacent pairs}.
\end{cases} 
\end{equation*}

    \begin{proof}
        Let $V(G)=\{u_1,u_2,...,u_n\}. $ Consider a vertex $u_i$ such that $deg(u_i)=\delta(G)$.Let $N_{S(G)}(u_i)=\{u_{i_1},u_{i_2},...,u_{i_\delta(G)}\}$ in $G \boxdot H$. Then, $(G \boxdot H)-({u_i} \cup N_{S(G)}(u_i))$ is disconnected and each connected component has at least $m>g+1$ vertices. 
        Hence,$\kappa(G \boxdot H) \leq \delta(G)+1$. Since, G is connected it implies $(G \boxdot H)$ is connected $\kappa(G \boxdot H)=\delta(G) \leq \kappa_g(G \boxdot H)$ but, since $ \kappa_g(G \boxdot H)=\delta(G) $iff  $g=0$. It implies $\delta(G)< \kappa_g(G \boxdot H)$. Hence, $\delta(G)+1\leq \kappa_g(G \boxdot H)$. Therefore, $\kappa_g(G \boxdot H)=\delta(G)+1$. Hence, the required claim is proved.\\
Case I: $A$ does not contain any adjacent vertices. Define $S=A   \cup \{ \cup_{i\in A}H^{i}\}$, $|S|= |A|(m+1)$ $(G \boxdot H)- S$ is disconnected and has at least $((k+1)(m+1)+\delta(G) $ vertices, $\kappa_g(G \boxdot H) \leq |S|$.By definition of $G \boxdot H$ there exists $L$ such that $\kappa_{g}(G \boxdot H)=|L|$.Claim:$|L\cap V(G)| \geq |A|$.Suppose the statement does not hold. And since $G-L$ is disconnected, there exists a connected component $C_1$ of $G-L$ with at most $k $ vertices.Then,  $C_1 \cup \{ \cup_{i\in V(C_1)}H^{i}\}$ is connected component of $(G\boxdot H)-L$ with atmost $k(m+1)$ vertices contradiction. Since,$L \cap V(G)  \subseteq Y$ also the $H$ copy corresponding vertices in $L \cap V(G)$ will be contained in $L$. $|L\cap V(G)|(m+1) \leq |L|$ hence, $|A|(m+1) \leq |L \cap V(G)|(m+1) \leq |L|$.This proves, $\kappa_g(G*H)=|A|(m+1)$.\\
 Case 2:When $A$ contains adjacent vertices. 
    By Case $1$  $\kappa_g(G \boxdot H) \leq |S|\leq|S|+|A'|$. Since, $A$  contains adjacent vertices,the edges between them correspond to subdivision vertices in the product graph, all such subdivision vertices must also be removed, otherwise it contradicts the existence of $g+1$  vertices$(g>0)$.
    Therefore,$|S|+|A'| \leq \kappa_g(G \boxdot H). $

     \end{proof}
\end{therm}
\begin{lem}
    Consider two connected graphs $G$  and $H$. Then, $\kappa_g(G\boxminus H)=1$ iff  $\kappa(G)= 1$ , for $0\leq g\leq m-1$.
    \begin{proof}
        Let $\kappa_g(G \boxminus H)=1$ then there exists a vertex $u$ such that $(G \boxminus H)-u$ is disconnected and each connected component has more than $g +1$ vertices. We prove, $u \in G$, equivalently we show $u$ is not a subdivision vertex or $u\not \in H$. Removing a subdivision vertex only removes the middle point of path $u-w_{uv}-v$ which cannot disconnect $G\boxminus H$.Also, removing a vertex from $H$ will never disconnect $G\boxminus H$. Hence, $u\in  G$ implies 
        $\kappa(G)=1$. Conversely , $\kappa(G)=1=\{u\}$, removing $u$ will break all paths that went through $u$, since $G \boxminus H$ preserves the structure of $G$,$G \boxminus H -u$ is disconnected and has at least  $m$ vertices.  This proves the claim.
        
    \end{proof}  
\end{lem}
\begin{lem}
     Let G, H be a nontrivial connected graph with $n$  and $m$ vertices respectively, then $\kappa_0(G \boxminus H)=2$ if $\kappa(G)\geq 2$.  
     \begin{proof}
         Given, $\kappa(G) \geq 2 $ removing a vertex will not disconnect $G$ and since $G\boxminus H$ preserves the structural property of $G$ removing a vertex will not disconnect $G\boxminus H $. $2\leq \kappa(G \boxminus H)=\kappa_0(G \boxminus H).$The subdivision vertex $w_{uv}$ can always be disconnected by removing $u,v$. $\kappa(G \boxminus H)=\kappa_0(G\boxminus H)\leq2$.
     \end{proof}
\end{lem}
\noindent Note: $\kappa_0(G\boxminus H)=1 $ iff $\kappa(G)=1$ and if $\kappa(G)\geq 2 $ then $\kappa_0(G\boxminus H)=2$. Further, we investigate the case for $g>0$.
\begin{therm}
    Let $G,H$ be two connected graphs with $n,m $vertices respectively then\\
    (i) If $1\leq g \leq m-1$ then
\begin{equation*}
\begin{cases}
  \kappa_g(G \boxminus H) =1 & \hspace{-9.2cm}\text { iff }   \kappa(G)=1. \\
  \kappa_g(G \boxminus H) =
\begin{cases}
    2 & \text { if }   \kappa(G)=2 \text{ and } \kappa(G) \text{has adjacent vertices.} \\
     3 & \text { if }   \kappa(G)=2 \text{ and } \kappa(G) \text{has non-adjacent vertices.}
\end{cases}
\\
\kappa(G \boxminus H)=3& \hspace{-9.2cm} \text { if } \kappa(G)\geq3.
\end{cases}
\end{equation*}  
\begin{proof}
    By Lemma [7] $\kappa_g(G\boxminus H)=1$iff $\kappa(G)$. \\
   Let $\kappa(G)=|T|=\{u_i,u_j\}$ such that $u_i\not \sim u_j$.The graph $(G\boxminus H)-T$ is disconnected and has atleast $m$ vertices so, $\kappa_g(G\boxminus H)\leq 2$. Since, $\kappa(G) \neq 1$ , $1< \kappa_g(G \boxminus H) \leq 2$.\\
   Suppose $u_i \sim u_j$ removing $u_i, u_j$ will isolate the subdivision vertex $w_{u_iu_j}$ and since $g>0$ we also remove $w_{uv}$ , hence
   $3\leq \kappa_g(G \boxminus H)$.
Clearly, $(G\boxminus H)-\{u_i,u_j,w_{u_iu_j}\}$ is a disconnected graph with $m$ vertices i.e., $(g+1)$ vertices, $\kappa_g(G\boxminus H)\leq 3$.  \\
 Given, $\kappa(G)\geq 3$. Let $V(G)=\{u_1,u_2,...,u_n\}$ pick $u_i \sim u_j$, and the subdivision vertex $w_{u_iu_j}$ $(G\boxminus H)-\{u_i,u_j,w_{u_iu_j}\}$ is a disconnected graph with $m$ vertices i.e., $(g+1)$ vertices, $\kappa_g(G\boxminus H)\leq 3.$  Since, $G\boxminus H$ preserves the structure of $G$ and $g>0.$ We have to atleast disconnect $G$ to disconnect $G \boxminus H$. $\kappa_g(G \boxminus H) \leq 3.$
\end{proof}
(ii) If $k (km+1) \leq g+1 \leq (k+1)+km+\delta(G) m+k+2$, then \begin{equation*}
   \kappa_g (G \diamond H)=
   \begin{cases}
     |X|  &\text{ when } X \text{ does not contain any adjacent vertices. } \\
    |X|+|X'|(m+1) & \text{ when } X \text{ contains $X'$adjacent   vertices }.
\end{cases} 
\end{equation*} 
where $X$ is defined to be the minimum cut set and each connected component of $G-X$   has at least $k+1 $ vertices.
\begin{proof}
    Case I: $X$ does not contain any adjacent vertices. Since, $(G \boxminus  H) $ preserves the structural property of $G$ and by definition of $X$ , $(G \boxminus  H)-X$ is disconnected and each connected component has  $(k+1)+km+\delta(G)m+k+2 $ vertices i.e., at least $(g+1)$ vertices and hence, $\kappa_g(G \boxminus  H) \leq |X|$.
      It suffices to prove $\kappa_g(G \boxminus  H) \geq |X|$.
    Let $\kappa_g(G \boxminus  H)=|L|$ such that $(G \boxminus  H)-L$ is not connected and each connected component of $(G \boxminus  H)-L$ has at least $g+1$ vertices.\\
    Claim: $|L|\geq |X|$. \\
    Assume to contrary $|L|< |X|$. 
    $G-L$ is not connected otherwise, we will get a component with fewer than $g+1 $ vertices, a contradiction. This further violates the assumption that $X$ is a minimum cutset of G. Therefore,$|L|\geq |X|$. Hence,$\kappa_g(G \boxminus  H)=|X|$.\\
        Case 2: When $X$ has adjacent vertices. From the definition of $X$, $G-X $ is not connected and each connected component of $G-X$ has at least $(k+1)$ vertices. Let $|X|=x$ and $|X'|=x'$ define $ S= X \cup V(He_{ij}) \cup (the subdivision vertex)$, $|S|= x+x'm+x'$,$(G \boxminus  H)-S$ is not connected and each connected component has atleast $(k+1)+km+\delta(G) m+k+2$ vertices i.e., at least $(g+1)$ vertices which implies $\kappa_g(G \boxminus  H) \leq |S| $. 
        Removing less than $|X|+|X'|(m+1)$ will disconnect the graph into components having less than $m$ vertices. Hence, $|X|+|X'|(m+1) \leq \kappa_g(G \boxminus  H) $.
\end{proof}
\end{therm}
\subsection{Generalized corona product.}
Generalized corona product $\left(G \circ \bigwedge_{i=1}^{n} H_i\right)$:Let \( G \) be a graph with vertex set $ V(G) = \{v_1, v_2,..,  v_n\} $, and let  $\{H_1, H_2, \ldots, H_n\} $ be a collection of arbitrary graphs. The generalized corona of $G$ with $ \{H_1, H_2, .., H_n\} $ denoted by $G \circ \bigwedge_{i=1}^{n} H_i$
is the graph obtained by taking one copy of $G$  and $n$ disjoint copies of the graphs  $H_1, H_2, ..., H_n$  and joining the vertex $v_i \in V(G)$ to  every vertex of $H_i $, for all $i = 1, 2, ..., n $.
\\
\begin{therm}Let $G$ be a connected graph of order $n$ and $H$ be a connected graph of order $m$.
    (i) If $0< g \leq m_i -1$ then $\kappa_g(G \circ \wedge_{i=1}^{n} H) =1$ where $m_i=min\{ m_1,m_2,...,m_n\}.$
    \begin{proof}
        Let the vertex set of $G$ be $V(G) = \{u_1, u_2, \ldots, u_n\}$.Let $u_i \in V(G)$ be such that the corresponding copy of $H_i$ in $G \circ \wedge_{i=1}^{n} H$ has the minimum number of vertices among all attached graphs. Then, the removal of $u_i$ from $G \circ \wedge_{i=1}^{n} H$ disconnects the graph. Since $0 \leq g \leq m_i - 1$, each resulting component has at least $g + 1$ vertices. Hence, $\kappa_g(G \circ \wedge_{i=1}^{n} H) \leq 1$. As $(G \circ \wedge_{i=1}^{n} H)$ is connected, it follows that $\kappa_g(G \circ \wedge_{i=1}^{n} H) = 1$.
    \end{proof}
    (ii) If $k(m_i+1)< g+1 \leq (k+1)(m_i+1)$ then, $\kappa_g(G \circ \wedge_{i=1}^{n} H)\leq|X|(m_i+1)$ where $X$ is minimum diconnecting set of $G$ such that $G-X$ has atleast $(k+1)$ vertices.
    \begin{proof}
        From the definition of the set $X$, the graph $G - X$ is disconnected, and each of its connected components contains at least $k + 1$ vertices. Let $|X| = x$, and without loss of generality, assume $X = \{u_{i_1}, u_{i_2}, ..., u_{i_x}\}$. Define the set  $S = X \cup V(H^{u_i})$ for all $u_{i} \in X$.Then, $|S| = (m_i + 1)x$. It is clear that removing $S$ from $(G \circ \wedge_{i=1}^{n} H)$ disconnects the graph. Moreover, each connected component of $(G \circ \wedge_{i=1}^{n} H)- S$ has at least $(k + 1)(m_i + 1)$ vertices. Therefore, $\kappa_g(G \circ \wedge_{i=1}^{n} H) \leq |S| = (m_i + 1)x = |X|(m_i + 1)$.
    \end{proof}
\end{therm}
\section{Rooted graph Product}
   The rooted product of $G$ and $H$ is a graph denoted by  $H(G)$, whose vertex and edge sets are defined as follows:\\
$V(H(G)) =  \{(g,h): g \in V(G)$ and  $ h \in V(H) \}$.  \\
$E(H(G))= \{ (g,h) (g,h ^\prime): g=g^\prime,hh^\prime  \in E(H), \text{ or } gg^\prime \in E(G), h=h^\prime =\text{root vertex of graph $H$} \}$.

\begin{therm}
    Let $G $ and $H$ be two arbitrary graphs with $n,m$ vertices respectively $m \neq 1$.\\
    (i) If $0\leq g \leq m-2$ then $\kappa_g(H(G))=1$.
    \begin{proof}
        Consider $V(G)=\{v_1,v_2,...,v_n\}$ and $V(H)=\{u_1,u_2,...,u_n\}$. Let the rooted vertex be $u_i$ in $H$ and $(u_i,v_i)$ be corresponding rooted vertex in $H(G)$. Clearly, $H(G)-(u_i,v_i)$ is disconnected. By the condition,  $0\leq g \leq m-2$ each connected component in $H(G)-(u_i,v_i)$ has at least $g+1$ vertices, and therefore $\kappa_g(H(G))\leq 1$. Since $H(G)$ is connected, implies $\kappa_g(H(G))=1.$ 
    \end{proof}
    (ii) $km< g+1 \leq (k+1)m$ then $\kappa_g(H(G))=|A|m$ where $A$ is defined as the minimum cutset and  $G-A$   has at least $k+1 $ vertices in each component.
    \begin{proof}
        Let $|A|=a$ from the definition of $A$, $G-A$ is not connected and each connected has atleast $k+1$ vertices. Define $U= A \cup V(r,H_i)$, $|U|=am$. $H(G)- U$ is disconnected and each connected component has at least $(k+1)m$ vertices, therefore, $\kappa_g(H(G))\leq|A|m$. To prove the other inequality, let $\kappa_g(G \diamond H)=|T|$ such that $(G \diamond H)-T$ is not connected and each connected component of $(G \diamond H)-T$ has at least $g+1$ vertices.\\ Claim:$|T \cap V(G) | \geq |A|.$ Let if possible $|T \cap V(G) | < |A|.$ $G-T$ is disconnected so there exists a component $C_1$ with atmost $k$ vertices . Hence, $C_1 \cup V(r,H_i)$is a connected component in $(H(G))-T$  with fewer than $g+1$ vertices contradiction, implies $|T\cap V(G) | \geq |A|.$ \\Let $|T \cap V(G) | =\{u_{a_1},...,u_{a_z}\}$ since,$T \cap V(G)  \subseteq T$ also the $H$ copy corresponding to each rooted vertices in $T\cap V(G)$ will be contained in $T$.$|T|\geq |T\cap V(G)|m+ \geq |A| m=|U|$.\end{proof}
\end{therm}
\begin{therm}
    Let $G$, $H$ be two arbitrary graphs. Then,
    \begin{enumerate}
    \item For $0 \leq g \leq m-2$.
    \begin{enumerate}[label=\alph*)] 
        \item  If $\lambda(G)=1 \implies \lambda_g(H(G))=1$.
        \item  If $\lambda (G) >1 \implies $ $ \lambda_g(H(G))=min\{\lambda(G),deg(v^r)\} $ where $v^r$ denotes the rooted  vertex. \end{enumerate}
        \end{enumerate}
        \begin{proof}
        Let $\lambda(G)=1$, there exists $e \in E(G)$ such that $G-e$ is disconnected. The structure of $H(G)$ assures $H(G)-e$ is not connected, and each connected component of $H(G)-e$ has at least $m$ vertices. Hence, $\lambda_g(H(G)) \leq 1$. since, $H(G)$ is connected $1\leq \lambda_g(H(G))$. This proves the claim.\\
        Let $\lambda(G)> 1$. Let $T\subseteq E(G)$ be the minimum edge cut of $G$ i.e., $|T|=\lambda (G)$. By the definition of $H(G)$, $H(G)-T $ is disconnected with at least $m$ vertices, which implies $\lambda_g(H(G) \leq |T|=\lambda(G)$. Let $v^r$ be the rooted vertex, $H(G)-E[v^r,H] $ is disconnected with atleast $m-1$ vertices, which implies $\lambda_g(H(G)) \leq deg(v^r)$. Hence, $\lambda_g(H(G)) \leq min \{ \lambda(G),deg(v^r)\}$. To prove the reverse inequality,if $\lambda(G) \leq deg(v^r)$, then by definition of $H(G)$, it preserves the structure of $G$ hence $\lambda(G) \leq \lambda_g(H(G))$. If $deg(v^r) \leq \lambda(G)$ deleting $deg(v^r)-1$ edges leaves the graph $H(G)$ connected. Therefore, $min \{ \lambda(G),deg(v^r)\} \leq \lambda_g(H(G))$. This validates the claim.
    \end{proof} 
    2. If $k(m-2) < g+1 \leq (k+1)m$ then $\lambda_g(H(G))=\lambda_k(G) $ where $\lambda_k(G)$ is the minimum edge cut whose removal disconnects the graph with at least $(k+1)$ vertices.
    \begin{proof}
        Let $\lambda_k(G)=|S|$ where $S \subseteq E(G)$ such that $G-S$ not connected with at least $k+1$ vertices. $H(G)-S$ is disconnected with at least $(k+1)m$ vertices. $\lambda_g(H(G)) \leq |S|.$ Let $\lambda_g(H(G))=T$ we show $|T| \geq |S|$. Claim: $|T \cap E(G)|\geq |S|$. Let $|T \cap E(G)< |S|$. Then, $G-T$ has a connected component with at least $k$ vertices. Hence, $H(C_1)$ is disconnected component of $H(G)$ with at most $km$ vertices, contradiction. $\lambda_g(H(G))=|T| \geq |T \cap E(G)|\geq |S|=\lambda_k(G).$ 
    \end{proof}

\end{therm}
    
    \bibliographystyle{plain}
\bibliography{Refrences}
\end{document}